\documentclass[a4paper]{article}

\usepackage[cp1251]{inputenc} 
\usepackage[T2A]{fontenc} 

\usepackage[english]{babel}
\usepackage{amsmath, amsfonts, amssymb, amsthm}
\usepackage{floatflt}
\usepackage{multicol}
\usepackage{amsfonts}
\usepackage{caption}
\usepackage{xcolor}

\usepackage{pdfpages} 
\usepackage[12pt]{extsizes} 

\usepackage{graphicx} 

 \newtheorem{theorem}{Theorem}
\newtheorem*{theoremA}{Theorem A}
\newtheorem*{theoremB}{Theorem B}
\newtheorem*{theoremC}{Theorem C}
\newtheorem*{theoremD}{Theorem D}
\newtheorem{lemma}{Lemma}

\numberwithin{equation}{section}

\DeclareCaptionLabelSeparator{dot}{. }
\title{The estimation of the length of a convex curve in two-dimensional Alexandrov space}
\author{Alexander ~A.~Borisenko }

\begin{document}
\date{}
\maketitle

\begin{center}
 It is proved the generalization of Toponogov theorem about the length of the curve
  in two-dimensional Riemannian manifolds in the case of two-dimensional Alexandrov spaces. \\[1pt]
  \end{center}
  
\hangindent=2cm\hangafter=-2\noindent

Let $R$ be an Alexandrov space of curvature $\ge c$ homeomorphic to a disc 
(see~\cite[p.~308]{Alex06}). 
Suppose $G$ is a domain in $R$ that is  bounded by a rectifiable curve $\gamma$.
Denote by $\tau(\gamma_1)$ the integral geodesic curvature (the swerve) of a subarc $\gamma_1$ of $\gamma$~\cite[p.~309]{Alex06}.
A curve $\gamma$ is called \textbf{ $\lambda$-convex } with $\lambda>0$ if 
any subarc $\gamma_1$ of  $\gamma$ satisties 
  \begin{equation}\label{def_lambda_convex}
  \frac{\tau(\gamma_1)}{s(\gamma_1)} \ge \lambda > 0, 
  \end{equation}
 where $s(\gamma_1)$ is the length of    $\gamma_1$.
 For regular curves in a two-dimensional Riemannian manifold 
this condition is equivalent to the assumption 
that the geodesic curvature at each point of this curve is $\ge \lambda > 0$.
In the general case the condition (\ref{def_lambda_convex}) allows $\gamma$ to have a corner points. 

 We prove the following theorem. 
\begin{theorem}\label{main}
Let $G$ be a domain homeomorphic to a disc and $G$ lies in a two-dimensional Alexandrov space of curvature $\ge c$
(in the sense of Alexandrov). 
\begin{itemize}
\item[\textnormal{I.}]
If the boundary curve $\gamma$ of $G$ is $\lambda$-convex and $c +\lambda^2 > 0$, then
the length $s(\gamma)$ of $\gamma$ satisfies
 \begin{enumerate}
 \item for $c = 0$  \begin{equation}\label{c_0}  s(\gamma) \le \frac{2\pi}{\lambda}; \notag   \end{equation}
 \item for $c > 0$  \begin{equation}\label{c_greater_0}  s(\gamma) \le \frac{  2\pi \sqrt{c} }{  \sqrt{c +\lambda^2}   };  
 \notag \end{equation}
 \item for $c < 0$  \begin{equation}\label{c_less_0}   s(\gamma) \le \frac{  2\pi \sqrt{-c} }{  \sqrt{c +\lambda^2}   }.  
 \notag \end{equation}
 \end{enumerate}
 \item[\textnormal{II.}]
All these  inequalities  attain equalities  if and only if the domain $G$ is a disc on the plane of constant curvature $c$. 
\end{itemize}
 \end{theorem}

This theorem is the generalization of Toponogov theorem~\cite{Top63} about the length of the curve
  in two-dimensional Riemannian manifolds.
We need the following statements to prove Theorem~\ref{main}.
 
\begin{theoremA}\textnormal{ (A.~D.~Alexandrov~\cite[p.~269]{Alex06}) }
A metric space with inner metric of curvature $\ge c$ homeomorphic to a sphere is isometric to
 a closed convex surface in a simply connected space of constant curvature $c$.
 \end{theoremA}

\begin{theoremB}\textnormal{ (A.~V.~Pogorelov~\cite[pp.~119-167, 267, 320-321]{Pog73},~\cite{Bor19})   } 
Closed isometric convex surfaces in three-dimensional Euclidean and spherical spaces are equal up to a rigid motion. 
 \end{theoremB}

\begin{theoremC}\textnormal{ (A.~D.~Milka~\cite{Milka80})  }
Closed isometric convex surfaces in three-dimensional Lobachevsky space are equal up to a rigid motion.
 \end{theoremC}

\begin{theoremD}\textnormal{ (W.~Blaschke~\cite{Blaschke56}) }
 Let $\gamma$ be a closed embedded $C^2$ regular  curve in  Euclidean  plane.
 \begin{itemize}
\item[\textnormal{I.}] If  the curvature $k$ of $\gamma$   at each its point $P$ satisfies  
\begin{equation} k \ge \lambda > 0, \notag   \end{equation}
then the curve belongs to the disc that is bounded by the circle of the radius $R=1/ \lambda$ tangent to the curve at point $P$.
 \item[\textnormal{II.}]  If the curvature $k$  of $\gamma$  at each its point $P$ satisfies  
\begin{equation} 0 \le k \le  \lambda, \notag   \end{equation}
then the circle of the radius  $R=1/ \lambda$ tangent to the curve at the point P 
belongs to the domain $G$ that is bounded by the curve $\gamma$.
\end{itemize}
 \end{theoremD}

The Theorem D is true if the condition for the curve's curvature  $k$  is substituted with
 the same condition for the specific curvature $ \frac{\tau(\gamma_1)}{s(\gamma_1)}$ for any  arc.

\begin{lemma}\label{lemm_BL}
 Let $\gamma$ be a closed embedded rectifiable curve in Euclidean plane. 
\begin{itemize}
\item[\textnormal{I.}] If for any subarc $\gamma_1$ of $\gamma$ 
the specific curvature $ \frac{\tau(\gamma_1)}{s(\gamma_1)}$ satisfies 
\begin{equation} \frac{\tau(\gamma_1)}{s(\gamma_1)} \ge \lambda > 0, \notag \end{equation}
then the curve $\gamma$ belongs to the disc that is bounded by the circle of the radius $R=1/ \lambda$
tangent to the support straight line at a point $P \in \gamma $ .
\item[\textnormal{II.}] If for any subarc $\gamma_1$ of $\gamma$ 
the specific curvature $ \frac{\tau(\gamma_1)}{s(\gamma_1)}$ satisfies 
\begin{equation} 0 \le \frac{\tau(\gamma_1)}{s(\gamma_1)} \le \lambda, \notag \end{equation}
then the circle of the radius $R=1/ \lambda$ tangent to the curve at a point P  
belongs to the domain $G$ that is bounded by the curve $\gamma$.
\end{itemize}
\end{lemma}

\begin{proof}
I. In this case the support function $h(\phi)$, $0\le \phi \le 2\pi$ of the curve $\gamma$ is $C^{1,1}$ regular 
and a.e. it satisfies the equation
\begin{equation} h+h'' = R, \;\; 0\le R \le \frac{1}{\lambda},\notag \end{equation}
where $R$ is a radius of curvature for $\gamma$. Therefore 
\begin{equation} h(\phi) = \int\limits_0^\phi R(\sigma)\sin(\phi-\sigma)\,d\sigma \notag \end{equation}
and the proof coincides with Blaschke proof~\cite{Blaschke56}. \\
II. The radius-vector $r(s)$ of the curve $\gamma$ is $C^{1,1}$ regular vector function. 
Fix the initial point $P_0$ on $ \gamma$ and denote by $e_1$ the unit tangent vector of $\gamma$ at $P_0$ and 
  by $e_2$ the unit normal vector of $\gamma$ at $P_0$.
Let $P(s)$ be the point on $ \gamma$ such that the length of the arc $\gamma(s)=P_0P(s)$ equals to $s$.
The function $\tau(s) = \tau\left(\gamma(s)\right)$ is the integral geodesic curvature of the arc $\gamma(s)$ and 
$\tau(s) \le \lambda s$. Therefore
\begin{equation} \label{r_prime}
r'(s) = \cos \tau(s) \, e_1 + \sin \tau(s) \, e_2. 
 \end{equation}
If we compare (\ref{r_prime}) with the equation for the circle of radius $\frac{1}{\lambda}$,  we obtain the proof.
\end{proof}

 H.  Karcher proved the generalization of Blaschke theorem in spherical space $\mathbb{S}^2$
  and in Lobachevsky space $\mathbb{H}^2$  for regular curves~\cite{Karcher68}.
 We formulate Lemma \ref{spher_curve} for the case when the curvature $\mathbb{S}^2$ equals to $1$ and the curvature 
  $\mathbb{H}^2$  is equal to $-1$. The Lemma  \ref{spher_curve} is true for the planes of any constant curvature $c$ 
  and the proof is the same. 
  
 \begin{lemma}\label{spher_curve}
Let $\gamma$ be a closed embedded rectifiable curve in $\mathbb{H}^2$ or $\mathbb{S}^2$.
\begin{itemize}
\item[\textnormal{I.}] If the specific curvature satisfies 
\begin{equation} \frac{\tau(\gamma_1)}{s(\gamma_1)} \ge \coth R_0 = \lambda, \notag \end{equation}
for any subarc $\gamma_1$ of $\gamma$ in $\mathbb{H}^2$, then the curve $\gamma$ belongs to the disc 
that is bounded by the circle of radius $R_0$  tangent to the support straight line of  $\gamma$ at a point $P \in \gamma$. 
\item[\textnormal{II.}] If the specific curvature satisfies 
\begin{equation} \frac{\tau(\gamma_1)}{s(\gamma_1)} \ge \cot R_0 = \lambda, \notag \end{equation}
for any subarc $\gamma_1$ of $\gamma$ in $\mathbb{S}^2$, then the curve $\gamma$ belongs to the disc  
that is bounded by the circle of radius $R_0$  tangent to the support straight line of  $\gamma$ at a point $P \in \gamma$.
\end{itemize}
 \end{lemma}
 
 \begin{proof}
The curve $\gamma$ is a closed convex curve. 
At any point $P $ of $\gamma$ there exists a support straight line (geodesic line in the plane of constant curvature).

 I. $\gamma \in \mathbb{H}^2$. 
Let $S$ be a circle of radius $R_0$ tangent to the support straight line of  $\gamma$  from the side containing  $\gamma$.
Assume that the center of the circle $S$ is the origin of the coordinate system in the Cayley-Klein model of Lobachevsky plane
and also it is the origin for support function $h$ of curve $\gamma$. 
The support function $h$ is $C^{1,1}$ regular and a.e. the radius of   curvature  $R$ of $\gamma$  equals
\begin{equation}   
R = \frac{ g + g''}{ \left( 1-  \frac{ (g')^2}{1+g^2}  \right)^{3/2}  }, 
\end{equation}
where $g(h) = \tanh h$ is the support function for the curve $\widetilde \gamma$, 
and $\widetilde \gamma$ is the image of $\gamma$ under the geodesic map $\mathbb{H}^2$ into $\mathbb{E}^2$~\cite{BorDrach15}~\cite{Drach14}. 
The radius of curvature $\widetilde R$ of  $\widetilde \gamma$  is a.e. equal to
\begin{equation}  
\widetilde R = R \left( 1-  \frac{ (g')^2}{1+g^2}  \right)^{3/2}, \, 0 \le \tilde R  \le R. 
\end{equation}
 The image of the circle $S$ under the geodesic map is the circle $\widetilde S$ in Euclidean plane $\mathbb{E}^2$ 
with the center in the origin of Cartesian orthogonal coordinate system.
The curvature of $\widetilde S$ equals to $\coth R_0$. 
From Lemma \ref{lemm_BL} (I.) it follows that $\widetilde \gamma$ belongs to the disc being bounded by the circle $\widetilde S$. 
Appling the inverse geodesic transformation, 
we obtain that the curve $\gamma$ belongs to the disc that is bounded by the circle $S$ in Lobachevsky plane $\mathbb{H}^2$.

II.  $\gamma \in \mathbb{S}^2$.
Let $\overline \gamma$ be the polar to $\gamma$ curve  in $\mathbb{S}^2$. 
The radius vector of $\overline \gamma$ is $C^{1,1}$ regular and its curvature is  $\le \tan R_0$ a.e.
Let $ P_0$ be a point on $  \gamma$ and 
$\overline S$ be a circle of the radius $\pi/2-R_0$ tangent to $\overline \gamma$ at the point $\overline P_0$. 
The  curvature of this circle is equal to $\tan R_0$. 
The center $\overline O$ of $\overline S$ is the south pole of the sphere. 
 Consider the geodesic map of the sphere $\mathbb{S}^2$ into the plane tangent to $\mathbb{S}^2$ at the point $\overline O$.
The curve $\overline \gamma$ is mapped to the curve $\widetilde {\overline \gamma} \in \mathbb{E}^2$ and
the circle $\overline S$ is maped into the circle $\widetilde {\overline S}$ of the curvature $\tan R_0$. 
The curvature $\widetilde {\overline k} (\widetilde {\overline \gamma}) \le \overline k (\overline \gamma) \le \tan R_0$.
From Lemma \ref{lemm_BL} (II.) it follows that
 the circle $\widetilde {\overline S}$ belongs to the domain that is bounded by the curve $\widetilde {\overline \gamma}$. 
Appling the inverse geodesic transformation, we obtain that
 the circle $\overline S$ belongs to the domain bounded by $\overline \gamma$
  and the polar curve $\gamma$ belongs to the disc bounded by the polar circle $S$ of the radius $R_0$. 
\end{proof}

\textit{Proof of the Theorem \ref{main}}.
 Let $G_1$ and $G_2$ be two copies of the domain $G$. 
Let us glue the domains $G_1$ and $G_2$ along their boundary curves $\gamma_1$ and $\gamma_2$
 by isometry between these curves. 
We obtain a manifold $F$ homeomorphic to the two-dimensional sphere with the intrinsic metric. 
Since the sum of the integral geodesic curvatures of any two identified arcs of the boundary curves is non-negative, 
from the Alexandrov gluing theorem \cite[p. 318]{Alex06} it follows that $F$ is Alexandrov space of curvature $\ge c$.
By Theorem A this manifold can be isometrically embedded as a closed convex surface $F_1$ 
in the simply-connected space $M^3(c)$ of constant curvature $c$. 
From Theorem B and C it follows that up to the rigid motion this surface is unique.

By \textit{plane domains} we will understand domains on totally-geodesic two-dimensional surfaces
in spaces of constant curvature; similarly we will call geodesic lines in these spaces as \textit{lines}. 

 Perform the reflection of the surface $F_1$ with respect to a plane $\pi$ passing through three points on $\gamma$ 
that do not belong to a line. 
We will get the mirrored surface $F_2$. 
The domains $G_1$ and $G_2$ are mapped to the domains $\widetilde G_1$ and $\widetilde G_2$ on $F_2$; 
the curve $\gamma$ is mapped to $\widetilde \gamma$. 
But $G_1$ is isometric to $G_2$ and $\widetilde G_2$ is isometric to $\widetilde G_1$.
Let us reverse the orientation of the domains $\widetilde G_1$, $\widetilde G_2$.
Then the surface $F_2$ will be isometric to $F_1$ and they will have the same orientation. 
By Theorems B and C the surface $F_1$ can be mapped into the surface $F_2$ by a rigid motion of the ambient space. 
But the three points of the curve $\gamma$ are fixed under this rigid motion.
Thus it follows that this motion is the identity mapping and, moreover, 
the curve $\gamma$ coincides with the curve $\widetilde \gamma$.
Such situation is possible only when the curve $\gamma$ is a plane curve and
 it is the boundary of a convex cup isometric to the domain $G$. 
Recall that the \textit{convex cup} is a convex surface with a plane boundary $\gamma$ such that the surface is a graph 
over a plane domain $\overline G$ enclose by $\gamma$. 
Note that, since $\gamma$ is a convex curve on the plane, 
then the integral geodesic curvature of any arc of the curve $\gamma$ is 
non-negative viewed both as a curve on the cup and as a curve on a plane \cite{Bor17}.

 Let us show that the integral geodesic curvature of any arc of $\gamma$ calculated on $G$ is not less than
  the corresponding integral geodesic curvature of it that is calculated on the cup $G$.
This means that $\gamma$ as a boundary curve of $\overline G$ is also $\lambda$-convex.

Recall that the intristic curvature $\omega(D)$ of a Borel set $D$ on a convex surface in a space of constant curvature $c$ is
\begin{equation}
\omega(D) = \psi(D) + c F(D), \notag
\end{equation}
where $\psi(D)$ is the extrinsic curvature, 
$F(D)$ is the area of $D$ \cite[p. 397]{Alex06}.
Consider a closed convex surface $M$ bounded by $G$ and the plane domain $\overline G$, and
a surface $\overline M$ made up from the double-covered domain $\overline G$.

The intrinsic curvature concentrated on $\gamma$ equals 
\begin{equation}
\omega(\gamma) = \tau_\gamma(G)+\tau_\gamma(\overline G), \notag
\end{equation}
where $ \tau_\gamma(G)$, $\tau_\gamma(\overline G)$ are the integral geodesic curvatures of $\gamma$ computed in $G$ 
and $\overline G$ respectively.

 Since $F(\gamma)=0$, we have
 \begin{equation}
\psi_M(\gamma) = \tau_\gamma(G)+\tau_\gamma(\overline G), \notag
\end{equation}
\begin{equation}
\psi_{\overline M}(\gamma)= 2\tau_\gamma(\overline G). \notag
\end{equation}
 
 From the definition of the extrinsic curvature \cite[p. 398]{Alex06} it follows that 
 $\psi_{\overline M}(\gamma) \ge \psi_{M}(\gamma)$ because each plane supporting to $M$ at a point of $\gamma$ is
 also supporting to $\overline M$.
 Thus we obtain $\tau_\gamma(\overline G) \ge \tau_\gamma(G)$. 
 Moreover, this inequality holds for any subarc of $\gamma$ as well.

I. The curve $\gamma$ is a $\lambda$-convex curve lying in the plane of constant curvature $c$.
From Lemmas \ref{lemm_BL} and \ref{spher_curve} it follows that 
the curve $\gamma$ belongs to the  disc bounded by the circle of radius $R_0$. 
The curvature and the length $s$ of these circle equals
 \begin{enumerate}
 \item  for $c = 0$, $\lambda = \frac{1}{R_0}$,  $s = 2\pi R_0$; 
 \item for $c > 0$,    $\lambda = \sqrt{c}\cot \sqrt{c} R_0$,  $s = 2\pi \sin \sqrt{c} R_0$;
 \item for $c < 0$,   $\lambda = \sqrt{-c} \coth \sqrt{-c} R_0$, $ s = 2\pi \sinh \sqrt{-c} R_0$.
 \end{enumerate}
The curve $\gamma$ on the plane of constant curvature $c$ bounds the convex domain $G$.
 It follows that the length of $\gamma$ satisfies 
  \begin{align}\label{s_gamma}
   s(\gamma) &\le
  \begin{cases}
   \frac{2\pi}{\lambda}       & \text{if } c = 0 \\
  \frac{2\pi \sqrt{c}}{ \sqrt{c + \lambda^2}}       & \text{if } c > 0 \\
   \frac{2\pi \sqrt{-c}}{ \sqrt{c + \lambda^2}}      & \text{if } c < 0  
  \end{cases}
   \end{align}
 
 II. Suppose that there is equality in (\ref{s_gamma}). 
Then the domain $\overline G$ is a disc bounded by the circle $\gamma$. 
Furthermore, $\tau_\gamma(\overline G) = \tau_\gamma(G)$ and 
the intrinsic curvature of $\gamma$ satisfies 
$\omega_M(\gamma) = \omega_{\overline M}(\gamma) = 2\tau_\gamma(\overline G)$
and the extrinsic curvature for any subarc $\gamma_1$ of $\gamma$ satisfies 
\begin{equation}\label{extr_curv}
\psi_M(\gamma) = \psi_{\overline M}(\gamma).
\end{equation}

 It follows that the surface $M$ and $\overline M$ coincide, $M$ is the double-covered disk and then $G$ is a disk.
 If $M$ doesn't coincide with $\overline M$ then there exists  the set of  a positive measure 
 of supporting planes to $\overline M$ along $\gamma$, that  are not planes of support to $M$.
It follows that the extrinsic curvatures of $M$ and  $\overline M$ along $\gamma$ don't coincides.
This contradicts the equality (\ref{extr_curv}). The Theorem \ref{main} is proved. 


   B.Verkin Institute for Low
Temperature Physics and Engineering of the National Academy of Sciences of Ukraine,
Kharkiv, 61103, Ukraine\\
\textsc {  \textit  {E-mail address: } }  aborisenk@gmail.com

\end{document}